\documentclass[12pt, letterpaper, twoside]{article}
\usepackage[utf8]{inputenc}

\date{November 8, 2017}
\usepackage{graphicx}
\usepackage{amssymb}
\usepackage{amsthm}
\usepackage{mathtools}
\newtheorem{definition}{Definition}[section]
\newtheorem{theorem}[definition]{Theorem}
\newtheorem{lemma}[definition]{Lemma}
\newtheorem{remark}[definition]{Remark}

\newtheorem{corollary}[definition]{Corollary}

\newtheorem{proposition}[definition]{Proposition}
\begin{document}

\title{Uniform Probability and Natural Density of Mutually Left Coprime Polynomial Matrices over Finite Fields}


\author{Julia Lieb          
}


\maketitle

\begin{abstract}
We compute the uniform probability that finitely many polynomials over a finite field are pairwise coprime and compare the result with the formula one gets using the natural density as probability measure. It will turn out that the formulas for the two considered probability measures asymptotically coincide but differ in the exact values. Moreover, we calculate the natural density of mutually left coprime polynomial matrices and compare the result with the formula one gets using the uniform probability distribution. The achieved estimations are not as precise as in the scalar case but again we can show asymptotic coincidence.
\end{abstract}

\section{Introduction}
Polynomial matrices over finite fields play an important role in various mathematical areas, e.g. for the investigation of discrete-time linear systems \cite{rb}, \cite{Fu-He15} or in the theory of convolutional codes \cite{ros}. For many of these applications, coprimeness conditions for the considered matrices are essential, \cite{Fu-He15}.

A polynomial matrix $D\in\mathbb F[z]^{n\times m}$ is called left prime if there exists $X\in\mathbb F[z]^{m\times n}$ with $DX=I$, where $I$ denotes the identity matrix. It is easily shown that this is equivalent to the condition that the fullsize minors of $D$ are coprime; see e.g. \cite{you}. In this paper, we will need another characterization of left primeness, namely that $D$ has to be of full row rank for every $z\in\overline{\mathbb F}$, which clearly is equivalent to the fact that it can be completed to a unimodular matrix, i.e. to a matrix with nonzero constant determinant. That it is possible to characterize left primeness by this last condition is part of the famous Quillen-Suslin theorem, also known as Serre conjecture \cite{lam}, which was formulated in 1957 for polynomial matrices in several variables $z_1,\hdots,z_k$. Already in 1958, Seshadri \cite{ses} proved its correctness in principal ideal domains and therefore, in the cases $k=1$ and $k=2$. The final proof for the general case followed in 1976 \cite{quill}, \cite{sus}.\\
We use the one-dimensional version of this theorem to compute the probability of left primeness for specially structured polynomial matrices using two different probability measures, namely uniform probability and natural density. For the case $n=1$, i.e. for matrices consisting only of one row, the probability of left primeness coincides with the probability of coprimeness for polynomials, which was computed in \cite{gar} to be equal to $1-t^{m-1}$, where $t:=|\mathbb F|^{-1}$. For matrices of arbitrary sizes, Guo and Yang \cite{guo} computed the natural density of left primeness to be equal to $\prod_{j=m-n}^{m-1}(1-t^j)$, using techniques from \cite{maze}, where this computation was done for integer matrices. Unfortunately, their proof contains a mistake. This has already been noticed by Micheli and Schnyder \cite{reto},\cite{gm}. In \cite[Problem 4.2, Theorem 4.4]{gm}, this problem is solved in a far more general context. The author computes densities over integrally closed subrings of global function fields using the definition of density given in \cite{gr}. The used strategy could also be found in \cite{fg} where the density of coprime algebraic integers of a number field is calculated.

In Theorem 9 of \cite{l}, the probability that a matrix of the form $[D_1\ D_2]\in\mathbb F[z]^{m\times 2m}$ with $\deg(\det(D_i))=n_i\in\mathbb N$ is left prime, i.e. that $D_1\in\mathbb F[z]^{m\times m}$ and $D_2\in\mathbb F[z]^{m\times m}$ are left coprime, was calculated. 
It turns out that the obtained formula, namely $1-t^m+O(t^{m+1})$ for $t\rightarrow 0$, asymptotically coincides with the formula for the natural density of left primeness for an arbitrary polynomial matrix from $\mathbb F[z]^{m\times 2m}$, computed in \cite{guo} -  respectively \cite{gm} - to be equal to $\prod_{j=m}^{2m-1}(1-t^j)$.

According to Proposition 10.3 of \cite{Fu-He15}, the property of $N$ matrices from $\mathbb F[z]^{m\times m}$ to be mutually left coprime is equivalent to the left primeness of a specially structured matrix from $\mathbb F^{(N-1)m\times mN}$. In \cite{hjL}, the uniform probability of mutual left coprimeness was calculated for polynomials with fixed degrees, i.e. for $m=1$, where mutual left coprimeness and pairwise coprimeness coincide. This result was generalized in \cite{l}, obtaining a probability of $1-\sum_{y=2}^{m+1}\binom{N}{y}t^m+O(t^{m+1})$ for the probability of mutual left coprimeness for $N$ matrices from $\mathbb F[z]^{m\times m}$ whose degrees of the determinant are fixed. In this article, we firstly improve the estimation for the case $m=1$ and secondly, compute the natural density of mutual left coprimeness in the cases $m=1$ and $m\in\mathbb N$. It will turn out that the formulas for uniform probability distribution and natural density asymptotically coincide in all computed cases.

The case $m=1$, i.e. natural density of pairwise coprimeness, was already considered in \cite{guo2}. However, the proof there contains the same mistake as mentioned before in the context of arbitrary rectangular matrices \cite{guo}. We will show a way to fix this problem in our article.

The paper is structured as follows. In Section 2, we provide some basic definitions, properties and formulas, which we will need in the following sections. Section 3 deals with the case $m=1$, i.e.  with uniform probability and natural density of pairwise coprime polynomials. It turns out that the obtained asymptotic expressions for uniform probability and natural density coincide. After that, in Section 4, we prove our main result, Theorem \ref{main}, which provides an asymptotic formula for the natural density of mutually left coprime polynomial matrices. Finally, we compare this result with the uniform probability that polynomial matrices are mutually left coprime and could again observe asymptotical identicalness.

\section{Preliminaries}
\subsection{Coprimeness of Polynomial Matrices}
In this subsection, we will provide some basic definitions and properties concerning polynomial matrices over an arbitrary field $\mathbb F$. Throughout this paper, $\overline{\mathbb F}$ should denote the algebraic closure of $\mathbb F$.

\begin{definition}\ \\
A polynomial matrix $Q\in\mathbb F[z]^{m\times m}$ is called \textbf{nonsingular} if $\det(Q(z))\not\equiv 0$. It is called \textbf{unimodular} if $\det(Q(z))\neq 0$ for all $z\in\overline{\mathbb F}$, i.e. if $\det(Q(z))$ is a nonzero constant. This is the case if and only if $Q$ is invertible in $\mathbb F[z]^{m\times m}$. Hence, one denotes the group of unimodular $m\times m$-matrices over $\mathbb F[z]$ by $GL_m(\mathbb F[z])$.
\end{definition}

\begin{definition}\ \\
A polynomial matrix $H\in\mathbb F[z]^{p\times m}$ is called a \textbf{common left divisor} of $H_i\in\mathbb F[z]^{p\times m_i}$ for $i=1,\hdots,N$ if there exist matrices $X_i\in\mathbb F[z]^{m\times m_i}$ with $H_i(z)=H(z)X_i(z)$ for $i=1,\hdots,N$. It is called a \textbf{greatest common left divisor}, which is denoted by $H=\operatorname{gcld}(H_1,\hdots,H_N)$, if for any other common left divisor $\tilde{H}\in\mathbb F[z]^{p\times\tilde{m}}$ there exists $S(z)\in\mathbb F[z]^{\tilde{m}\times m}$ with $H(z)=\tilde{H}(z)S(z)$.\\
A polynomial matrix $E\in\mathbb F[z]^{p\times m}$ is called a \textbf{common left multiple} of $E_i\in\mathbb F[z]^{m_i\times m}$ for $i=1,\hdots,N$ if there exist matrices $X_i\in\mathbb F[z]^{p\times m_i}$ with $X_i(z)E_i(z)=E(z)$ for $i=1,\hdots,N$. It is called a \textbf{least common left multiple}, which is denoted by $E=\operatorname{lclm}(E_1,\hdots,E_N)$, if for any other common left multiple $\tilde{E}\in\mathbb F[z]^{\tilde{p}\times m}$, there exists $R(z)\in\mathbb F[z]^{\tilde{p}\times p}$ with $R(z)E(z)=\tilde{E}(z)$.\\
One defines a \textbf{(greatest) common right divisor}, which is denoted by $\operatorname{gcrd}$, and a \textbf{(least) common right multiple}, which is denoted by $\operatorname{lcrm}$, analoguely.\\
Note that $\operatorname{gcd}$ and $\operatorname{lcm}$ are only unique up to multiplication with an unimodular matrix but this does not matter for our further considerations.
\end{definition}

\begin{definition}\ \\
Polynomial matrices $H_i\in\mathbb F[z]^{p\times m_i}$ are called \textbf{left coprime} if there exists $X\in\mathbb F[z]^{m\times p}$ such that $H=\operatorname{gcld}(H_1,\hdots,H_N)$ satiesfies $HX=I_p$. In particular, one polynomial matrix $H\in\mathbb F[z]^{p\times m}$ is called \textbf{left prime} if there exists $X\in\mathbb F[z]^{m\times p}$ with $HX=I_p$. Analoguely, one defines the property to be \textbf{right coprime} or \textbf{right prime}, respectively. Note that in the case $p=m$, right primeness and left primeness are equivalent to the property to be unimodular.
\end{definition}

\begin{theorem}\cite[Theorem 2.27]{Fu-He15}\ \\
The polynomial matrices $H_i\in\mathbb F[z]^{p\times m_i}$ are left coprime if and only if\\
$\operatorname{rk}(H_1(z),\hdots,H_N(z))=p$ for all $z\in\overline{\mathbb F}[z]$.
\end{theorem}

\begin{remark}\label{hauc}
\item[(a)]
Left coprimeness of $H_i\in\mathbb F[z]^{p\times m_i}$ is equivalent to left primeness of the matrix $(H_1,\hdots,H_N)$.
\item[(b)]
A rectangular matrix $H\in\mathbb F[z]^{p\times m}$ with $p\leq m$ is left prime if and only if its $p\times p$-minors are coprime; see e.g. \cite{you}. 
\end{remark}

\begin{definition}\ \\
Nonsingular polynomial matrices $D_1,\hdots, D_N\in\mathbb F[z]^{m\times m}$ are called \textbf{mutually left coprime} if for each $i=1,\hdots,N$, $D_i$ is left coprime with $\operatorname{lcrm}\{D_j\}_{j\neq i}$.
\end{definition}

\begin{theorem}\cite[Proposition 10.3]{Fu-He15}\label{mutcrit} \\
Nonsingular polynomial matrices $D_1,\hdots, D_N\in\mathbb F[z]^{m\times m}$ are mutually left coprime if and only if
$$\mathcal{D}_N:=\left[\begin{array}{cccc}
D_1 & D_2 &  & 0 \\ 
 & \ddots & \ddots &  \\ 
0 &  & D_{N-1} & D_N
\end{array}\right]$$
 is left prime.
\end{theorem}

\subsection{Probability Distributions and Basic Counting Formulas}

To compute the probability that a mathematical object has a special property, it is necessary to count mathematical objects. Therefore, in the following, we restrict our considerations to a finite field $\mathbb{F}$ with cardinality $|\mathbb F|$.\\
Firstly, $\mathbb F$ should be endowed with the uniform probability distribution that assigns to each field element the same probability
$$t=\frac{1}{|\mathbb{F}|}.$$ 
In addition to computing probabilties with the uniform distribution, which is only defined for finite sets, we will compare these results with the results one gets using another definition of probability, namely the concept of natural density as defined in \cite{guo} for infinite sets:

\begin{definition}\label{nd}\ \\
To enumerate $\mathbb F[z]$, assign the number $k=\sum_{i=0}^{\infty}a_i(\frac{1}{t})^{i}$ to the polynomial $f_k(z)=\sum_{i=0}^{\infty}a_iz^{i}$. In particular, $f_0\equiv 0$. Moreover, let $M_n$ be the set of tuples $(D_1,\hdots,D_N)\in(\mathbb F[z]^{l\times m})^N$ for which the entries of $D_i$ are elements of the set $\{f_0,\hdots,f_n\}$ for $i=1,\hdots,N$. The natural density of a set $E\subset (\mathbb F[z]^{l\times m})^N$ in $(\mathbb F[z]^{l\times m})^N$ is defined as $\lim_{n\rightarrow\infty}\frac{|E\cap M_n|}{|M_n|}$.
\end{definition}

Moreover, for later computations, we will need the following lemmata, which provide well-known formulas for the determination of cardinalities.

\begin{lemma}\cite[S. 455]{lini}\label{glcar}\ \\
For $1\leq r\leq \min(k,n)$, denote by $N(k,n,r)$ the number of matrices from $\mathbb F^{k\times n}$ that have rank $r$. Then, it holds
$$N(k,n,r)=t^{-nr}\prod_{i=n-r+1}^n(1-t^i)\cdot \prod_{i=0}^{r-1}\frac{t^{i-k}-1}{t^{-(i+1)}-1}.$$
In particular, the number of invertible $n\times n$-matrices over $\mathbb F$ is equal to
$$|Gl_n(\mathbb F)|=t^{-n^2}\prod_{j=1}^n(1-t^j).$$
\end{lemma}

\begin{lemma}\label{ine}(Inclusion-Exclusion Principle)\ \\
Let $A_1,\hdots,A_n$ be finite sets and $X=\bigcup_{i=1}^n A_i$. For $I\subset\{1,\hdots,n\}$, define $A_I:=\bigcap_{i\in I}A_i$. Then, it holds
$$|X|=\sum_{\emptyset\neq I\subset\{1,\hdots,n\}}(-1)^{|I|-1}|A_I|.$$
\end{lemma}

\begin{lemma}\label{2coprime}\cite{gar} \\
The probability that $N$ monic polynomials $d_1,\hdots,d_N\in\mathbb F[z]$ with $\deg(d_i)=n_i\in\mathbb N$ for $i=1,\hdots, N$ are coprime is equal to $1-t^{N-1}$.
\end{lemma}

\begin{lemma}\label{anzirred}\cite{ch}\ \\
The number of monic irreducible polynomials in $\mathbb F[z]$ of degree $j$ is equal to 
$$\varphi_j=\frac{1}{j}\sum_{d\mid j}\mu(d)t^{-j/d}=\frac{1}{j}t^{-j}+O(t^{-(j-1)}),$$
where for $n\in\mathbb N$, $\mu(n):=\begin{cases}(-1)^{|\{p\in\mathbb P\ |\ p\mid n\}|}, & n\ \text{square-free}\\0, & \text{otherwise}\end{cases}$.
\end{lemma}

\begin{lemma}\label{bin}\ \\
For $M\in\mathbb N$, it holds
$$\sum_{k=1}^M(-1)^k\frac{1}{k!(M-k)!}=-\frac{1}{M!}.$$
\end{lemma}

\begin{proof}
Using the binomial formula, one obtains
$$0=(1-1)^M=\sum_{k=0}^M\binom{M}{k}(-1)^k\quad \Leftrightarrow\quad 0=\sum_{k=0}^M(-1)^k\frac{1}{k!(M-k)!}.$$
\end{proof}

\section{Counting Pairwise Coprime Polynomials}\label{pair}

According to \cite[Corollary 3]{hjL}, the probability that $N$ monic polynomials $d_i\in\mathbb F[z]$ with $\deg(d_i)=n_i\in\mathbb N$ for $i=1,\hdots,N$ are pairwisely coprime is 
\begin{align}\label{j1}
1-\frac{N(N-1)}{2}\cdot t+O(t^2)
\end{align}
if  $1/t$ tends to infinity. 
The method used in \cite{hjL} to prove this result has the advantage that in principle, it is possible to compute the coefficients of $t^j$ for $j\geq 2$ in this asymptotic expansion with the same procedure. But when $j$ is increasing, the computational effort for doing this becomes very large.
In the following, we want to improve this estimation by additionally computing the coefficient of $t^2$. Prior to this, we need to introduce some notation, which was also used in \cite{hjL}.

First, a more general setup should be considered.
Let $n:=(n_1,\ldots,n_N)\in\mathbb N^N$ and $\Gamma$ be an undirected
graph with set of vertices $\mathcal{V}=\{1,\ldots,N\}$ and set of
edges $\mathcal{E}$, having cardinality $E:=|\mathcal{E}|$. The edges
of $\Gamma$ are denoted as $\widehat{ij}$, for suitable $i,j\in \mathcal{V}$ with
$i<j$. For every vertex $l\in \mathcal{V}$ let
\[\mathcal{E}_l:=\{\widehat{ij}\in\mathcal{E}\;|\; i=l\; \text{or}\; j=l\}\] 
denote the set of edges terminating at $l$.
Moreover, $\operatorname{gcd}$ and
$\operatorname{lcm}$ should denote the monic greatest common divisor and least common multiple, respectively. Let $X(n):=\{(d_1,\ldots,d_N)\ |\ d_i\in\mathbb F[z]\ \text{monic with}\ \deg(d_i)=n_i\}$
and $\Gamma(n):=\{(d_1,\ldots,d_N)\in X(n)\ |\ \operatorname{gcd}(d_i,d_j)=1\ \text{for}\ \
\widehat{ij}\in \mathcal{E}\}.$ Clearly, $|X(n)|=t^{-(n_1+...+n_N)}$. 
With each edge $\widehat{ij}$ of $\Gamma$ we associate a monic, square-free
polynomial $k_{\widehat{ij}}(z)\in \mathbb{F}[z]$. We refer to this as a polynomial
labeling of the graph and denote it by $\textbf{k}$. For each polynomial labeling and vertices
$l \in \mathcal{V}$, let
\[
K_l:=\operatorname{lcm}\{k_{\widehat{ij}}\;|\; \widehat{ij}\in \mathcal{E}_l\}.
\]
Then
$$M(n):=\{\textbf{k}\in\mathbb F[z]^E\;|\;k_{\widehat{ij}} \ \text{monic,
square-free for}\ \widehat{ij}\in \mathcal{E},\ \deg(K_l)\leq n_l,\ l\in \mathcal{V}\}$$
is the set of all polynomial labelings $\textbf{k}$ of $\Gamma$
satisfying the degree bounds $\deg (K_l)\leq n_l$ for all vertices $l$.
For each monic square-free polynomial $p$, let $\omega(p)$ denote the number of
irreducible factors of $p$.  
To achieve formula \eqref{j1} as well as our improvement, the following exact expression for the considered probability is used:
 
\begin{theorem}\label{graph}\cite[Theorem 5]{hjL}\ \\
The cardinality of $\Gamma(n)$ is
\begin{equation}\label{eq:result}
|\Gamma(n)|=t^{-(n_1+...+n_N)}\sum_{\textbf{k}\in M(n)}\prod_{\widehat{ij}\in \mathcal{E}}(-1)^{\omega(k_{\widehat{ij}})}\prod_{l=1}^N t^{\deg(K_l)}.
\end{equation}
\end{theorem}

In the case that all pairs of vertices of $\Gamma$ are connected by an edge, one obtains the probability that $N$ monic polynomials are pairwise coprime. Next, the preceding theorem is used to improve the estimation from formula \eqref{j1}:
%

\begin{theorem}\label{3}\ \\
Let $n_1,...,n_N\in\mathbb N$ and $N_1:=|l\in\{1,\ldots,N\}\ |\ n_l=1|$. Then, the probability that $N$ monic polynomials over $\mathbb F$ of degrees $n_1,\ldots,n_N$ are pairwise coprime is equal to
$$1-\frac{N(N-1)}{2}\cdot t+\frac{1}{24}(N-1)(N-2)(3N^2+11N-12N_1)\cdot t^2+O(t^{3}).$$
\end{theorem}

\begin{proof}
Let $G$ be a graph with $N$ vertices, which are pairwisely connected by an edge, i.e. the number of edges $E$ is $\frac{N(N-1)}{2}$. Furthermore, let $|G(n)|$ be the number of $N$-tuples of monic pairwise coprime polynomials over $\mathbb F$ of degrees $n_1,\ldots,n_N$. Using Theorem \ref{3} with $\Gamma=G$, one gets that the probability of pairwise coprimeness is equal to
$$|G(n)|\cdot t^{n_1+\cdots+n_N}=\sum_{\textbf{k}\in M(n)}\prod_{\widehat{ij}\in \mathcal{E}}(-1)^{\omega(k_{\widehat{ij}})}\prod_{l=1}^N t^{\deg(K_l)}.$$
Thus, to show the stated formula, we have to compute the series expansion in $t$ of this term till the coefficient of $t^2$.


We first sort the elements of $M(n)$ with respect to the degrees of the entries of the vector $\textbf{k}=(k_1,...,k_E)$.
To this end, for each vector of non-negative integers $\mathbf{g}:=(g_1,...,g_E)$ define
$M(n,\mathbf{g}):=\{\mathbf{k}\in M(n)\  |\ \deg(k_{m})=g_m\ for\ 1\leq m\leq E\}$. Let $A$ be the set of all $\mathbf{g}$ with $M(n,\mathbf{g})\neq \emptyset$. Note that the degree bounds for $M(n)$ ensure that $A$ is finite. 
One achieves:
\begin{eqnarray*}
|G(n)|=t^{-(n_1+...+n_N)}\sum_{\mathbf{g}\in A}\  \sum_{\mathbf{k}\in M(n,\mathbf{g})}\ \prod_{\widehat{ij}\in\mathcal{E}}(-1)^{\omega(k_{\widehat{ij}})}\prod_{l=1}^N t^{\deg(K_l)}.
\end{eqnarray*}
Starting with small values for the entries of $\mathbf{g}$, the first summands are computed. 
For $\mathbf{g}=(0,\ldots,0)$, i.e. $\mathbf{k}=(1,\ldots,1)$, one gets the summand 1 because of $\omega(1)=0$ and $K_l=1$ for $l=1,\ldots,N$. If $g_{m_0}=1$ for exactly one $1\leq m_0\leq E$ and $g_m=0$ for $m\neq m_0$, there are $|\mathbb F|=1/t$ possibilities for the linear polynomial $k_{m_0}$ and $E$ possibilities for the choice of $m_0$. Moreover, $\omega(k_{m_0})=1$, so that these summands have negative sign. 
As $k_{m_0}$ is relevant for exactly those $K_l$ for which its associated edge is terminating at $l$, there are exactly two $K_l$ which are of degree 1. Hence, the resulting sum of these terms is equal to $-E\cdot\frac{1}{t}\cdot t^{2}=-E\cdot t.$\\
Note that for all summands computed so far, every $\mathbf{k}$ lies in $M(n,\mathbf{g})$ since $\deg(K_l)\leq 1$ in all considered cases. Next look at the summands whose sum of the entries of $\mathbf{g}$ is equal to $2$. 
The absolute value of each summand with $g_{m_0}=2$ for exactly one $1\leq m_0\leq E$ and $g_m=0$ for $m\neq m_0$ is equal to $t^4$. They have negative sign if $k_{m_0}$ is irreducible and positive sign otherwise. Since the number of irreducible monic polynomials of degree $2$ is equal to the number of reducible monic polynomials of degree 2 (see e.g. Remark 1 of \cite{hjL}), it follows that these summands add up to zero. Hence, in this case, it does not matter whether $\mathbf{k}$ lies in $M(n,\mathbf{g})$ or not since this depends only on $m_0$ and not on $k_{m_0}$ itself.

Now consider the summands for which two entries of $\mathbf{g}$ are equal to one, and the other entries are equal to zero. 
If the corresponding edges of the nonzero entries have a vertex $l$ in common, the summand has the value 
\begin{align}\label{angle}
\sum_{\stackrel{\text{\scriptsize $k_1,k_2$ monic}}{\deg(k_m)=1}}t^{2+\deg(\operatorname{lcm}(k_{1},k_{2}))}&=
\sum_{\stackrel{\text{\scriptsize $k_1=k_2$ monic}}{\deg(k_m)=1}}t^3+\sum_{\stackrel{\text{\scriptsize $k_1\neq k_2$ monic}}{\deg(k_m)=1}}t^4=\nonumber\\
&=\frac{1}{t}\cdot t^3+\frac{1}{t}\cdot \left(\frac{1}{t}-1\right)\cdot t^4=2t^2-t^3
\end{align}
if $n_l\geq 2$ and $t^2$ if $n_l=1$ since, in this case, $k_1\neq k_2$ implies $\deg(K_l)=2>n_l$ and thus $\mathbf{k}\notin M(n,\mathbf{g})$ and the second sum of the preceding computation vanishes. For such an "angle", there are $N\cdot\binom{N-1}{2}$ possibilities, $N$ for the apex and $\binom{N-1}{2}$ for the two sides of the angle.

If those two edges are isolated, the summand has the value 
$$\sum_{\stackrel{\text{\scriptsize $k_1,k_2$ monic}}{\deg(k_m)=1}}t^4=t^2.$$
For this case, there are $\binom{N}{4}$ possibilities to choose the $4$ involved vertices and $3$ possibilities to connect two of them, pairwisely.

In summary, all summands whose sum of the entries of $\mathbf{g}$ is at most two contribute the value 
\begin{align}\label{2}
&1-\frac{N(N-1)}{2}\cdot t+\left(\binom{N-1}{2}(2(N-N_1)+N_1)+3\cdot\binom{N}{4}\right)\cdot t^2+O(t^3)=\nonumber\\
&1-\frac{N(N-1)}{2}\cdot t+\frac{(N-1)(N-2)}{8}(N^2+5N-4N_1)\cdot t^2+O(t^3).
\end{align}

If three entries of $\mathbf{g}$ are equal to $1$, where the corresponding edges form a triangle, and the other entries are equal to $0$, one gets
$$-\frac{1}{t}\cdot t^3-\frac{3}{t}\cdot\left(\frac{1}{t}-1\right)\cdot t^5-\frac{1}{t}\cdot\left(\frac{1}{t}-1\right)\cdot\left(\frac{1}{t}-2\right)\cdot t^6=-t^2+O(t^3).$$
Here the first summand of the left hand side of the equation 
gives the probability for the case that three, the second summand that two and the third summand that none of the three entries of $\mathbf{k}$ that contain a linear polynomial are identical. Moreover, there are $\binom{N}{3}$ possibilities for such a triangle. Adding these summands to \eqref{2}, one gets
\begin{align*}
1-\frac{N(N-1)}{2}\cdot t+\frac{(N-1)(N-2)}{8}(N^2+5N-4N_1-4N/3)\cdot t^2+O(t^3)=\\
1-\frac{N(N-1)}{2}\cdot t+\frac{(N-1)(N-2)}{24}(3N^2+11N-12N_1)\cdot t^2+O(t^3).
\end{align*}

It remains to show that the summands for all other values of $\mathbf{g}$ are $O(t^3)$, i.e. that
$$R(\mathbf{g}):=\sum_{\mathbf{k}\in M(n,\mathbf{g})}\prod_{l=1}^N t^{\deg(K_l)}=O(t^{3})$$
for every fixed $\mathbf{g}$ for which the sum of the entries of $\mathbf{g}$ is at least three and for which it does not hold that three entries of $\mathbf{g}$ are equal to $1$, where the corresponding edges form a triangle, and the other entries are equal to $0$.

To this end, define $\Gamma$ as any subgraph of $G$ and $E$ as the number of edges of $\Gamma$ and show the above estimation for $R(g)$ per induction with respect to $E$.

For $E=1$, note that $\mathbf{g}$ and $\mathbf{k}=k_{\widehat{12}}$ are scalar. Moreover, $K_{1}=K_{2}=k_{\widehat{12}}$. Therefore, $R(\mathbf{g})=0$ if $\mathbf{g}>\min(n_{1},n_{2})$ and otherwise
$$R(\mathbf{g})\leq\sum_{\mathbf{k}\ \text{monic},\ \deg(\mathbf{k})=\mathbf{g}}t^{2\deg(\mathbf{k})}=\left(\frac{1}{t}\right)^{\mathbf{g}}\cdot t^{2\mathbf{g}}=t^{\mathbf{g}}=O(t^{3})\ \text{for}\ \mathbf{g}\geq 3.$$
This computation starts with an inequality since the condition that $\mathbf{k}$ has to be square-free is dropped. The first equality follows from the fact that there are $(1/t)^{\mathbf{g}}$ monic polynomials of degree $\mathbf{g}$.

Next, we take the step from $E-1$ to E.
To this end, choose one of the smallest entries of $\mathbf{g}$ and denote it without loss of generality by $g_E$. Then, the edge with which $k_E$ is associated -- in the following denoted by $\widehat{ij}$ -- is taken away form the original graph and thus a graph with $E-1$ edges is achieved.
In the following, the index $(E-1)$ above an expression means that it belongs to a graph with $E-1$ edges; in the same way we use the index $(E)$. Similarly, $\mathbf{k}^{(E-1)}$ and $\mathbf{g}^{(E-1)}$ should denote the vectors consisting of the first $E-1$ entries of $\mathbf{k}$ and $\mathbf{g}$, respectively.
The degrees of the $K_l$ can never increase, when taking an edge away. Therefore, $\mathbf{k}\in M(n,\mathbf{g})$ implies $\mathbf{k}^{(E-1)}\in M^{(E-1)}(n,\mathbf{g}^{(E-1)})$.
Next we set
$$W_i:=\gcd(K^{(E-1)}_{i}, k_{E})\ \ \ \text{and}\ \ \ W_j:=\gcd(K^{(E-1)}_{j}, k_{E}).$$
Moreover, let
\begin{eqnarray*}
B_{v_i,v_j}^{(E-1)}:=\{\mathbf{k}^{(E-1)}\in M^{(E-1)}(n,\mathbf{g}^{(E-1)})\ |\ \deg(K^{(E-1)}_{i})=v_i,\  \deg(K^{(E-1)}_{j})=v_j\},\\
B_{v_i,v_j,w_i,w_j}^{(E)}:=\{\mathbf{k}\in M^{(E)}(n,\mathbf{g})\ |\ \mathbf{k}^{(E-1)}\in B_{v_i,v_j}^{(E-1)},\ \deg(W_i)=w_i,\ \deg(W_j)=w_j\}.
\end{eqnarray*}
It follows
$$R(\mathbf{g})\leq \sum_{v_i,v_j,w_i,w_j\leq\max(n_{i}, n_{j})}\ \sum_{\mathbf{k}\in B_{v_i,v_j,w_i,w_j}^{(E)}} \prod_{l=1}^N t^{\deg(K^{(E)}_l)}.$$
The number of summands in the first sum is finite and thus one only has to show that for any fixed $v_i,v_j,w_i,w_j$ the following is true:
$$\sum_{\mathbf{k}\in B_{v_i,v_j,w_i,w_j}^{(E)}} \prod_{l=1}^N t^{\deg(K^{(E)}_l)}=O(t^{3}).$$
To do this one computes
$$K_{i}^{(E)}=\operatorname{lcm}(K_{i}^{(E-1)}, k_{E})=\frac{K_{i}^{(E-1)}\cdot k_{E}}{W_i}.$$
Consequently, one has $\deg(K_{i}^{(E)})=\deg(K^{(E-1)}_{i})+g_E-w_i$ and $\deg(K_{j}^{(E)})=\deg(K^{(E-1)}_{j})+g_E-w_j$, analogously.
For $l\notin\{i,j\}$ it holds $K_l^{(E)}=K_l^{(E-1)}$ because nothing changes at the associated vertices. It follows:
\begin{eqnarray}\label{eq:final}
\sum_{\mathbf{k}\in B_{v_i,v_j,w_i,w_j}^{(E)}} \prod_{l=1}^N t^{\deg(K^{(E)}_l)}=\sum_{\mathbf{k}\in B_{v_i,v_j,w_i,w_j}^{(E)}} \prod_{l=1}^N t^{\deg(K^{(E-1)}_l)}\cdot t^{2g_E-w_i-w_j}.
\end{eqnarray}
Here, the product $\prod_{l=1}^N t^{\deg(K^{(E-1)}_l)}$ is independent of $k_{E}$ and $t^{2g_E-w_i-w_j}$ is independent of $\mathbf{k}$. 

Next, for $\mathbf{k}^{(E-1)}\in B_{v_i,v_j}^{(E-1)}$, an upper bound for the number of polynomials $k_{E}$ such that $\mathbf{k}\in B_{v_i,v_j,w_i,w_j}^{(E)}$ should be determined.
$\mathbf{k}^{(E-1)}$ uniquely determines $K^{(E-1)}_{i}$ and since $W_i$ is a divisor of $K^{(E-1)}_{i}$ of degree $w_i$, there are only finitely many possibilities for $W_i$. Define $C$ as this number of possibilities for $W_i$. One knows that $k_{E}$ has to be a multiple of $W_i$ of degree $g_E$. Thus, for each $W_i$ there are at most $t^{w_i-g_E}$ possibilities for $k_{E}$.
Using this and the fact that the product in (\ref{eq:final}) is independent of $k_{E}$, it follows for the expression in (\ref{eq:final}):
\begin{align*}
\sum_{\mathbf{k}\in B_{v_i,v_j,w_i,w_j}^{(E)}} \prod_{l=1}^N t^{\deg(K^{(E)}_l)}
&\leq  t^{2g_E-w_i-w_j}\cdot C t^{w_i-g_E} \sum_{\mathbf{k}^{(E-1)}\in B_{v_i,v_j}^{(E-1)}} \prod_{l=1}^N t^{\deg(K^{(E-1)}_l)}\\
&= C t^{g_E-w_j}\sum_{\mathbf{k}^{(E-1)}\in B_{v_i,v_j}^{(E-1)}} \prod_{l=1}^N t^{\deg(K^{(E-1)}_l)}\\
&\leq  C\cdot R(\mathbf{g}^{(E-1)})
\end{align*}
because $w_j\leq g_E$ since $W_j\mid k_{E}$.
Now we distinguish several cases:

\underline{Case 1}: The sum of the entries of $\mathbf{g}^{(E-1)}$ is at least three and it does not hold that three entries of $\mathbf{g}^{(E-1)}$ are equal to $1$, where the corresponding edges form a triangle, and the other entries are equal to $0$.\\
%
%
%
%
%
%
Then, $R(\mathbf{g}^{(E-1)})$ is $O(t^{3})$ per induction and we are done.

\vspace{1mm}
\underline{Case 2}: $\mathbf{g}^{(E-1)}$ has a component that is equal to zero.\\
Here, $g_E$ must be zero since it was chosen to be one of the smallest entries. 
Thus, $\Gamma^{(E-1)}$ and $\Gamma^{(E)}$ could be treated as being identical and hence, the conditions of case 1 are fulfilled.
Consequently, we are done, too. 
\vspace{1mm}

\underline{Case 3}: $\mathbf{g^{(E-1)}}=(1,1,1)$ and $\Gamma^{(E-1)}$ is a triangle.\\
This case can be avoided: It holds $\mathbf{g^{(E)}}=(1,1,1,1)$ since $\mathbf{g^{(E)}}=(1,1,1,0)$ would mean that $\Gamma^{(E)}$ is a triangle, too, because an edge $\widehat{ij}$ with labelling $k_{\widehat{ij}}=1$ could be treated like it would not exist. Therefore, one of the vertices of the triangle has an third edge which connects it with the additional vertex. Since all entries of $\mathbf{g^{(E)}}$ are identical, one can take away an arbitrary edge in our process of induction. If one takes away one of the edges which form the triangle, the resulting $\Gamma^{(E-1)}$ is not a triangle any more.

It remains to consider all possible cases for which the sum of the entries of $\mathbf{g}^{(E-1)}$ is smaller than three but the sum of the entries of $\mathbf{g}^{(E)}$ is at least three and it does not hold that three entries of $\mathbf{g}^{(E)}$ are equal to $1$, where the corresponding edges form a triangle, and the other entries are equal to $0$.
First, one considers $\mathbf{g}^{(E-1)}=(1,1)$ (case 4) and then $\mathbf{g}^{(E-1)}=2$ (cases 5 and 6).


\vspace{1mm}
\underline{Case 4}: $\mathbf{g}^{(E)}=(1,1,1)$ and $\Gamma^{(E)}$ is no triangle.

Case 4a: $\Gamma^{(E)}$ consists of three isolated edges:
$$R(\mathbf{g})\leq\left(\frac{1}{t}\right)^3\cdot t^6=t^3=O(t^3).$$

Case 4b: $\Gamma^{(E)}$ consists of an isolated edge and an angle (see \eqref{angle}):
$$R(\mathbf{g})\leq\frac{1}{t}\cdot t^2\cdot(2t^2-t^3)=O(t^3).$$

Case 4c: $\Gamma^{(E)}$ consists of three edges forming one line:
$$R(\mathbf{g})\leq\frac{1}{t}\cdot t^4+\frac{2}{t}\left(\frac{1}{t}-1\right)\cdot t^5+\frac{1}{t}\left(\frac{1}{t}-1\right)^2\cdot t^6=O(t^3).$$
The first summand covers the case that all linear polynomials in $\mathbf{k}$ are identical, the second summand the case that the polynomial of the edge in the middle coincides with one of the others and the third polynomial is different and the third summand the case that the polynomial in the middle is different from the other two polynomials.

Case 4d: $\Gamma^{(E)}$ consists of three edges that meet at one vertex:
$$R(\mathbf{g})\leq\frac{1}{t}\cdot t^4+\frac{3}{t}\left(\frac{1}{t}-1\right)\cdot t^5+\frac{1}{t}\left(\frac{1}{t}-1\right)\left(\frac{1}{t}-2\right)\cdot t^6=O(t^3).$$
The first summand covers the case that all linear polynomials in $\mathbf{k}$ are identical, the second summand the case that exactly two of them are identical and the third summand the case that all polynomials are different.

\vspace{1mm}
\underline{Case 5}: $\mathbf{g}^{(E)}=(2,1)$.\\
Since we are considering upper bounds for $R(\mathbf{g})$ in the following, we can drop the condition that the quadratic polynomials have to be square-free.

Case 5a: $\Gamma^{(E)}$ consists of two isolated edges:
$$R(\mathbf{g})\leq\left(\frac{1}{t}\right)^3\cdot t^6=O(t^3).$$

Case 5b: $\Gamma^{(E)}$ consists of an angle:
$$R(\mathbf{g})\leq\frac{1}{t}\left(\frac{1}{t}-1\right)\cdot t^5+\frac{1}{t^3}\cdot t^6=O(t^3).$$
The first summand covers the case that the linear polynomial divides the quadratic polynomial, the second summand the other case. 

\vspace{1mm}
\underline{Case 6}: $\mathbf{g}^{(E)}=(2,2)$.

Case 6a: $G$ consists of two isolated edges:
$$R(\mathbf{g})\leq\left(\frac{1}{t}\right)^4\cdot t^8=O(t^3).$$

Case 6b: $G$ consists of an angle:
$$R(\mathbf{g})\leq\frac{1}{t^2}\cdot t^6+\frac{1}{t^4}\cdot t^7=O(t^3).$$
The first summand covers the case that the two quadratic polynomials are identical, the second summand the other case.

It follows $R(\mathbf{g})=O(t^3)$ for every fixed $\mathbf{g}$ for which the sum of the entries of $\mathbf{g}$ is at least three and for which it does not hold that three entries of $\mathbf{g}$ are equal to $1$, where the corresponding edges form a triangle, and the other entries are equal to $0$. Consequently,
\begin{align*}
&|G(n)|=t^{-(n_1+\hdots+n_N)}\cdot\\
&\cdot \left(1-\frac{N(N-1)}{2}\cdot t+\frac{1}{24}(N-1)(N-2)(3N^2+11N-12N_1)\cdot t^2+O(t^{3})\right).
\end{align*}
\end{proof}
So far, we used the uniform probability distribution and fixed the degrees of the considered polynomials.
In the following, this result should be compared with the natural density of pairwise coprime polynomials. As mentioned in the introduction the natural density of pairwise coprime polynomials has already been computed in \cite{guo2} but the proof in that paper is not correct. In Remark \ref{fix}, we will show how to fix this problem.
In this section of our paper, we just cite the result and use it for comparison with the formula for the uniform probability distribution.

\begin{theorem}\label{pwe}\cite{guo2}\ \\
The natural density of $N$ polynomials $d_1,\hdots,d_N\in\mathbb F[z]$ to be pairwise coprime is equal to 
$$\prod_{j=1}^{\infty}\left((1-t^{j})^{N-1}(1+(N-1)t^{j})\right)^{\varphi_j}.$$
 \end{theorem}

\begin{corollary}\label{gem}\ \\
The natural density of $N$ polynomials $d_1,\hdots,d_N\in\mathbb F[z]$ to be pairwise coprime is equal to 
$$1-\binom{N}{2}t+\frac{1}{24}(N-1)(N-2)(3N^2+11N)t^{2}+O(t^{3}).$$
\end{corollary}

\begin{proof}
One has to show 
\begin{align*}
&\prod_{j=1}^{\infty}((1-t^{j})^{N-1}(1+(N-1)t^{j}))^{\varphi_j}=\\
&=1-\binom{N}{2}t+\frac{1}{24}(N-1)(N-2)(3N^2+11N)t^{2}+O(t^{3}).
\end{align*}
One uses the estimations $\varphi_j=\frac{1}{j}t^{-j}+O(t^{-(j-1)})$ as well as
\vspace{-3.7mm}
\begin{align*}
&F_j:=(1-t^{j})^{N-1}(1+(N-1)t^{j})=\\
&=1+\left(\binom{N-1}{2}-(N-1)^2\right)t^{2j}+\\
&+\left(\binom{N-1}{2}(N-1)-\binom{N-1}{3}\right)t^{3j}+O(t^{4j})=\\
&=1+\frac{N-1}{2}\cdot (N-2-2(N-1))t^{2j}+\\
&+(N-1)(N-2)\left(\frac{N-1}{2}-\frac{N-3}{6}\right)t^{3j}+O(t^{4j})\\
&=1-\binom{N}{2}t^{2j}+\frac{1}{3}N(N-1)(N-2)t^{3j}+O(t^{4j}).
\end{align*}
If one chooses $x$ times the term with exponent $-kj$ (for $k\geq 2$) expanding $\prod_{j=1}^{\infty}F_j^{\varphi_j}$, one gets a term of the form $C(N)\binom{\varphi_j}{x}t^{xkj}=O(t^{(k-1)xj})$ with $C(N)$ only depending on $N$. Thus, one is only interested in the case $k-1=x=j=1$ and in the case that one number from the set $\{k-1,x,j\}$ is equal to $2$ and the others are equal to $1$. In particular, the considered probability is 
\begin{align*}
&\underbrace{\left(1-\binom{N}{2}t^{2}+\frac{1}{3}N(N-1)(N-2)t^{3}\right)^{t^{-1}}}_{j=1,\ k\leq 3}\underbrace{\left(1-\binom{N}{2}t^{4}\right)^{\frac{1}{2}(t^{-2}-t^{-1})}}_{j=2,\ k\leq 2}+O(t^{3})=\\
&=\left(1-\underbrace{\binom{N}{2}t}_{k=2,\ x=1}+\underbrace{\frac{1}{3}N(N-1)(N-2)t^{2}}_{k=3,\ x=1}+\underbrace{\binom{t^{-1}}{2}\binom{N}{2}^2t^{4}}_{k=2,\ x=2}+O(t^{3})\right)\cdot\\
&\cdot\left(1-\binom{N}{2}t^{4}\cdot\frac{1}{2}t^{-2}+O(t^{3})\right)+O(t^{3})=\\
&=1-\binom{N}{2}t+\left(\frac{1}{3}N(N-1)(N-2)+\frac{N^2(N-1)^2}{8}-\frac{N(N-1)}{4}\right)t^{2}+O(t^{3})\\
&=1-\binom{N}{2}t+\left(\frac{1}{3}N(N-1)(N-2)+\frac{N(N-1)}{8}(N(N-1)-2)\right)t^{2}+O(t^{3})\\
&=1-\binom{N}{2}t+\left(\frac{1}{3}N(N-1)(N-2)+\frac{1}{8}(N-1)(N-2)(N+1)N \right)t^{2}+O(t^{3})\\
&=1-\binom{N}{2}t+\frac{1}{24}(N-1)(N-2)(3N^2+11N)t^{2}+O(t^{3}),
\end{align*}
which completes the proof of the corollary.
\end{proof}
Corollary \ref{gem} leads to the same asymptotic formula as Theorem \ref{3} with setting $N_1=0$, although different concepts of probability were used. This concordance could be explained in the following way: First, computing the natural density of pairwise coprimeness, those tuples of polynomials which contain a linear polynomial could be neglected. Moreover, the case that $d_i\equiv 0$ for some $i\in\{1,\cdots, n\}$ could be neglected and hence, considering monic polynomials does not change the probability because two polynomials are coprime if and only if the corresponding monic polynomials are coprime. Thus, all degree dependencies of the considered coefficients in the asymptotic expansion could be neglected. 
Hence, for sufficiently large $n_i$ the (uniform) probability is identical for all values $n_i\in\mathbb N$.
Therefore, choosing the polynomials randomly with $\deg(d_i)\leq n_i$, the probability could be regarded as identical for all values $n_i\in\mathbb N$ since the set of polynomials with $\deg(d_i)\leq n_i$ is a disjunct union of the sets of polynomials whose degree is a fixed value less or equal to $n_i$. But the sets defined by the condition $\deg(d_i)\leq n_i$ form a subsequence of $M_n$. Consequently, if one knows that the limit defining the natural density, i.e. $\lim_{n\rightarrow \infty}\frac{|E\cap M_n|}{|M_n|}$, exists, one could conclude that it is equal to the constant value for this subsequence. 
Note that, in this case, $\lim_{n\rightarrow\infty}\left(\frac{|E\cap M_n|}{|M_n|}-\left(1-\frac{N(N-1)}{2}\cdot t+\frac{(N-1)(N-2)(3N^2+11N-12N_1)}{24}\cdot t^2\right)\right)$ exists\\
and therefore, the coefficient of $O(t^3)$ cannot go to infinity for $n\rightarrow\infty$.

\section{Mutual Left Coprimeness of Polynomial Matrices}

The aim of this section is to compute the natural density of mutual left coprime polynomial matrices from $\mathbb F[z]^{m\times m}$ and compare it with the uniform probability that $N$ nonsingular polynomial matrices are mutual left coprime, which was estimated in \cite{l}.

\begin{theorem}\label{mut}\cite[Theorem 12]{l}\ \\
For $m, N\geq 2$ and $n_i\in\mathbb N$ for $i=1,\hdots,N$, the uniform probability that $D_i\in\mathbb F[z]^{m\times m}$ with $\deg(\det(D_i)=n_i$ for $i=1,\hdots, n$ are mutually left coprime is 
$$1-\sum_{y=2}^{m+1}\binom{N}{y}t^m+O(t^{m+1}).$$
\end{theorem}

In the following, we want to compare the preceding result with the formula one gets for the natural density. It will turn out that, as in the previous section, the problem of computing the natural density could be reduced to the calculation of the (uniform) probability that $N$ constant matrices are mutually left coprime.
To show this statement, i.e. to prove Theorem \ref{end}, we need the following proposition, which was proven in \cite{gm}. There, the problem of computing densities was considered in a more general setup, namely in the context of integrally closed subrings of global function fields. Specifying the more general statement of Theorem 2.2 from \cite{gm} to the case of polynomial rings over finite fields, one gets the following result:

\begin{proposition}\cite[Theorem 2.2]{gm}\label{gia}\ \\
Let $H:=\mathbb F[z]$ and let $p,q\in H[x_1,\hdots,x_d]$ be coprime polynomials. Define $\hat{M}_{c}:=\{y\in H\ |\ \deg(y)\leq c\}$ and\\
$A_g:=\{y\in H^d\ |\ p(y)\equiv q(y)\equiv 0 \mod f\ \text{for some}\ f\in H\ \text{with}\ \deg(f)\geq g+1\}$.
Then
$$\lim_{g\rightarrow\infty}\limsup_{c\rightarrow\infty}\frac{|A_g\cap \hat{M}_{c}|}{|\hat{M}_{c}|}=0.$$
\end{proposition}

For the proof of Theorem \ref{end}, we also need the following definition and lemma.

\begin{definition}\ \\
For $j\in\mathbb N$, denote by $W_j(N)$ the probability that  $\mathcal{K}_N:=\left[\begin{array}{cccc}
K_1 & K_2 & 0 & 0 \\ 
0 & \ddots & \ddots & 0 \\ 
0 & 0 & K_{N-1} & K_N
\end{array}\right]$ 
with $K_i\in(\mathbb F^j)^{m\times m}$ for $i=1,\hdots,N$ is of full row rank, i.e. that the matrices $K_i$ are mutually left coprime.\\
For $I\subset\{1,\hdots,N\}$, denote by $\mathcal{K}^{(I)}_{N-|I|}$ the matrix formed by the matrices from the set $\{K_i\ |\ i\notin I\}$.
\end{definition}

\begin{lemma}\label{ki}\ \\
If $\det(K_i)\neq 0$ for all $i\in I$, then $\mathcal{K}_N$ has full row rank if and only if $\mathcal{K}^{(I)}_{N-|I|}$ has full row rank.
\end{lemma}

\begin{proof}\ \\
Assume without restriction that $I=\{N-|I|+1,\hdots,N\}$ (otherwise permutate the matrices $K_i$). Since $\det(K_N)\neq 0$, the columns of $K_N$ form a basis of $(\mathbb F^{j})^m$ and adding appropriate linear combinations of the last $m$ columns to the $m$ preceding columns of $\mathcal{K}_N$ brings  $\mathcal{K}_N$ to the form $\left[\begin{array}{ccccc}
K_1 & K_2 & 0 & \hdots & 0 \\ 
0 & \ddots & \ddots & \ddots & 0 \\ 
0 & 0 & K_{N-2} & K_{N-1} & 0\\
 0 & 0  & 0 & 0 & K_{N} \\
\end{array}\right]$, which is not left prime if and only if the submatrix consisting of the first $(N-2)m$ rows is not left prime. One iterates this procedure with the matices $K_{N-1},\hdots,K_{N-|I|+1}$ and the result follows per induction.
\end{proof}

For the proof of the following Theorem, we expand and modify the idea of the proof for Theorem 1 in \cite{guo}.

\begin{theorem}\label{end}\  \\
The natural density of $N$ matrices $D_i\in\mathbb F[z]^{m\times m}$ for $i=1,\hdots,N$ to be mutually left coprime is equal to 
$\prod_{j=1}^{\infty}W_j(N)^{\varphi_j}$.
\end{theorem}

\begin{proof}\ \\
From Theorem \ref{mutcrit}, one knows that $D_1,\hdots,D_N$ are mutually left coprime if and only if the matrix
$$\mathcal{D}_N:=\left[\begin{array}{ccccc}
D_1 & D_2 & 0 & \cdots & 0 \\ 
0 & D_2 & D_3 & \ddots & \vdots \\ 
\vdots & \ddots & \ddots & \ddots & 0\\
 0 & \cdots & 0 & D_{N-1} & D_{N}
\end{array}\right] $$
is left prime. According to Remark \ref{hauc} (b), this holds if and only if the size $m(N-1)$ minors of $\mathcal{D}_N$ are coprime.

In the following, the notation of Definition \ref{nd} is used.
Let $M_n$ be the set of all tuples $(D_1,\hdots,D_N)\in(\mathbb F[z]^{m\times m})^N$ for which all entries of $D_i$ are contained in $\{f_0,\hdots,f_n\}$ for $i=1,\hdots,N$. Furthermore, let $\hat{P}$ be the set of all (monic) irreducible polynomials in $\mathbb F[z]$ and $P$ a finite subset of $\hat{P}$. Moreover, $E_P$ should denote the set of all tuples $(D_1,\hdots,D_N)\in(\mathbb F[z]^{m\times m})^N$ for which the gcd of all size $m(N-1)$ minors of $\mathcal{D}_N$ is coprime with all elements in $P$. Consequently, we are interested in the probability that $(D_1,\hdots,D_N)\in(\mathbb F[z]^{m\times m})^N$ lies in\\
$E:=\bigcap_P E_P$; i.e., for the natural density one has to determine $\lim_{n\rightarrow\infty}\frac{|E\cap M_n|}{|M_n|}$.

In a first step, one computes the probability that $(D_1,\hdots,D_N)\in M_n$ lies in $E_P$. To this end, one defines $f_P:=\prod_{f\in P}f$ and $d_P:=\deg(f_P)$. Next, consider the projection
\begin{align*}
M_n&\rightarrow M_n/(f_P)=\prod_{f\in P} M_n/(f)\\
(D_1,\hdots,D_N)&\mapsto (D_1,\hdots,D_N)/(f_P)=\prod_{f\in P}(D_1,\hdots,D_N)/(f),
\end{align*}
which applies the canonical projection modulo $f_P$ ($\mathbb F[z]\rightarrow \mathbb F[z]/(f_P)$) to each entry of $(D_1,\hdots,D_N)$. For $(D_1,\hdots,D_N)\in M_n$ holds:
\begin{align*}
&(D_1,\hdots,D_N)\in E_P\\
&\Leftrightarrow \forall f\in P\ \exists\ \text{size}\ m(N-1)\ \text{minor of $\mathcal{D}_N$ that is not divided by}\ f\\
&\Leftrightarrow \forall f\in P\ \exists\ \text{size}\ m(N-1)\ \text{minor of $\mathcal{D}_N$ nonzero in}\  (\mathbb F[z]/(f))^{m(N-1)\times mN}\\
&\Leftrightarrow \forall f\in P: \mathcal{D}_N/(f)\ \text{of full rank in}\ (\mathbb F[z]/(f))^{m(N-1)\times mN}\simeq(\mathbb F^{\deg{f}})^{m(N-1)\times mN},
\end{align*}
where $\mathbb F^{\deg{f}}$ denotes the field with $t^{-\deg(f)}$ elements.
Denote the probability that $\mathcal{D}_N/(f):=\mathcal{K}_N\in(\mathbb F^{\deg(f)})^{(N-1)m\times Nm}$ has full row rank by $W_f$. 

First, suppose that $t^{-d_P}$ divides $|\{f_0,\hdots,f_n\}|=n+1$, i.e. $n=bt^{-d_P}-1$ for some $b\in\mathbb N$. Then, one could write $\{f_0,\hdots,f_n\}=\{f_s(z)z^{d_P}+f_r(z)\ |\ 0\leq s\leq b-1,\ 0\leq r\leq t^{-d_P}-1\}$. One has $\{f_r\ |\ 0\leq r\leq t^{-d_P}-1\}\simeq\mathbb F[z]/(f_P)$ and $f_s(z)z^{d_P}+f_r(z)\mod f_P(z)=f_s(z)z^{d_P}\mod f_P(z)+f_r(z)=\hat{f}_s(z)+f_r(z)$ where $\hat{f}_s(z):=f_s(z)z^{d_P}\mod f_P(z)\in\mathbb F[z]/(f_P)$. Hence, for every fixed $s$ the canonical projection is bijective and on $\{f_0,\hdots,f_n\}$ it is $b$-to-one. In summary, one obtains
\begin{align*}
|E_P\cap M_n|&=b^{m^2N}\cdot \prod_{f\in P}t^{-m^2N\deg(f)}\cdot W_f=(bt^{-d_P})^{m^2N}\cdot \prod_{f\in P}W_f.
\end{align*}
Since $bt^{-d_P}=n+1$, i.e. $(bt^{-d_P})^{m^2N}=|M_n|$, it follows
$$\frac{|E_P\cap M_n|}{|M_n|}=\prod_{f\in P}W_f.$$
Now, suppose $n\in\mathbb N$ arbitrary. By division with remainder, we get $n+1=bt^{-d_P}+r$ with $0\leq r<t^{-d_P}$. One defindes $\hat{n}:=n+t^{-d_P}-r=(b+1)t^{-d_P}-1$. 
Since
\begin{align*}
\lim_{n\rightarrow\infty}\frac{|E_P\cap(M_{\hat{n}}\setminus M_{n})|}{|M_n|}&\leq \lim_{n\rightarrow\infty}\frac{|M_{\hat{n}}|-|M_{n}|}{|M_n|}=\\
&=\lim_{n\rightarrow\infty}\frac{(n+1+t^{-d_P}-r)^{m^2N}-(n+1)^{m^2N}}{(n+1)^{m^2N}}=0,
\end{align*}
one has
\begin{align*}
&\lim_{n\rightarrow\infty}\frac{|E_P\cap M_n|}{|M_n|}=\lim_{n\rightarrow\infty}\frac{|E_P\cap M_{\hat{n}}|-|E_P\cap(M_{\hat{n}}\setminus M_{n})|}{|M_n|}=\lim_{n\rightarrow\infty}\frac{|E_P\cap M_{\hat{n}}|}{|M_{n}|}\\
&=\lim_{n\rightarrow\infty}\frac{(n+t^{-d_P}-r+1)^N\prod_{f\in P}W_f}{(n+1)^N}=\prod_{f\in P}W_f.
\end{align*}
To estimate $W_f$, we show
%
that at least $2$ of the matrices $K_i$ have zero determinant if $\mathcal{K}_N$ is not of full row rank.

 If $N=2$, this clearly is true because $K_1$ and $K_2$ are left coprime if not both of them have zero determinant. For $N\geq 3$, assume without restriction that $\det(K_N)\neq 0$ (otherwise permutate the matrices $K_i$).
Per induction and using Lemma \ref{ki}, it follows that at least two of the matrices $K_1,\hdots,K_{N-1}$ have zero determinant, which gives us the desired result.


Define $H_f=\mathbb F[z]^{m^2N}\setminus E_f$.
Let $P_g$ be the set of all irreducible polynomials with degree at most $g$. Then $E_{P_g}\setminus E\subset\bigcup_{f\in\hat{P}\setminus P_g}H_f$ and consequently, 
\begin{align*}
\frac{|(E_{P_g}\setminus E)\cap M_n|}{|M_n|}\leq\frac{|(\bigcup_{f\in\hat{P}\setminus P_g}H_f)\cap M_n|}{|M_n|}
\end{align*}
As shown above, $\mathcal{D}_N\in H_f$ implies that there exist $i,j\in\{1,\hdots,N\}$ with $i\neq j$ and $\det(D_i)\equiv\det(D_j)\equiv 0 \mod f$.
We apply Proposition \ref{gia} with $d=m^2N$, $p=\det(D_1)$ and $q=\det(D_2)$ considered as polynomials in the polynomial entries of $D_1$ and $D_2$. Since $p$ has the entries of $D_1$ as variables and $D_2$ has the entries of $D_2$ as variables, the two polynomials have no common variable and are therefore coprime.
Moreover, write $n+1=ct^{-1}+w$ with $c,w\in\mathbb N_0$ and $w<t^{-1}$.
Defining $A_g$ as in Proposition \ref{gia}, one gets
\begin{align*}
\frac{|(\bigcup_{f\in\hat{P}\setminus P_g}H_f)\cap M_n|}{|M_n|}&\leq \binom{N}{2}\frac{|A_g\cap M_n|}{|M_n|}\leq \binom{N}{2}\frac{|A_g\cap \hat{M}_{c+1}|}{|\hat{M}_{c+1}|}\frac{|\hat{M}_{c+1}|}{|M_n|}\\
&\leq \binom{N}{2} t^{-1}\frac{|A_g\cap \hat{M}_{c+1}|}{|\hat{M}_{c+1}|} 
\end{align*}
since there are $\binom{N}{2}$ possibilities to choose $D_i$ and $D_j$ with $i\neq j$ and\\
 $|M_n|\geq t^{-(c+1)}=t\cdot|\hat{M}_{c+1}|$.
Therefore,
\begin{align*}
\lim_{g\rightarrow\infty}\limsup_{n\rightarrow\infty}\frac{|(E_{P_g}\setminus E)\cap M_n|}{|M_n|}&\leq
\lim_{g\rightarrow\infty}\limsup_{n\rightarrow\infty}\frac{|(\bigcup_{f\in\hat{P}\setminus P_g}H_f)\cap M_n|}{|M_n|}\\
&\leq\binom{N}{2} t^{-1} \lim_{g\rightarrow\infty}\limsup_{c\rightarrow\infty}\frac{|A_g\cap \hat{M}_{c+1}|}{|\hat{M}_{c+1}|}=0
\end{align*}
where the last equality follows from Proposition \ref{gia}.

Since $E\cap M_n=E_{P_g}\cap M_n\setminus ((E_{P_g}\setminus E)\cap M_n)$, one obtains
\begin{align*}
&\liminf_{n\rightarrow\infty}\frac{|E\cap M_n|}{|M_n|}\geq\liminf_{n\rightarrow\infty}\frac{|E_{P_g}\cap M_n|}{|M_n|}-\limsup_{n\rightarrow\infty}\frac{|(E_{P_g}\setminus E)\cap M_n|}{|M_n|}
\end{align*}
and hence
$$\lim_{g\rightarrow\infty} \liminf_{n\rightarrow\infty}\frac{|E\cap M_n|}{|M_n|}\geq\lim_{g\rightarrow\infty}\liminf_{n\rightarrow\infty}\frac{|E_{P_g}\cap M_n|}{|M_n|}=\lim_{g\rightarrow\infty}\lim_{n\rightarrow\infty}\frac{|E_{P_g}\cap M_n|}{|M_n|}$$
as well as
\begin{align*}
\limsup_{n\rightarrow\infty}\frac{|E\cap M_n|}{|M_n|}&\leq\limsup_{n\rightarrow\infty}\frac{|E_{P_g}\cap M_n|}{|M_n|}-\liminf_{n\rightarrow\infty}\frac{|(E_{P_g}\setminus E)\cap M_n|}{|M_n|}\\
&\leq\lim_{n\rightarrow\infty}\frac{|E_{P_g}\cap M_n|}{|M_n|}.
\end{align*}
It follows
\begin{align*}
\lim_{n\rightarrow\infty}\frac{|E\cap M_n|}{|M_n|}&=\lim_{g\rightarrow\infty}\lim_{n\rightarrow\infty}\frac{|E_{P_g}\cap M_n|}{|M_n|}=\lim_{g\rightarrow\infty}\prod_{f\in P_g}W_f=\lim_{g\rightarrow\infty}\prod_{j=1}^gW_j(N)^{\varphi_j}=\\
&=\prod_{j=1}^{\infty}W_j(N)^{\varphi_j}.
\end{align*}

\begin{remark}\label{fix}\ \\
As mentioned before, there is a problem with the proofs for the formulas of the natural density in \cite{guo} and \cite{guo2}. One has to show
\begin{align}\label{prop}
\lim_{g\rightarrow\infty}\limsup_{n\rightarrow\infty}\frac{|(E_{P_g}\setminus E)\cap M_n|}{|M_n|}\leq\lim_{g\rightarrow\infty}\limsup_{n\rightarrow\infty}\frac{|(\bigcup_{f\in\hat{P}\setminus P_g}H_f)\cap M_n|}{|M_n|}=0.
\end{align}
To do this, the authors of \cite{guo} and \cite{guo2} use the following chain of inequalities
\begin{align*}
&\limsup_{n\rightarrow\infty}\frac{|(\bigcup_{f\in\hat{P}\setminus P_g}H_f)\cap M_n|}{|M_n|}\leq \limsup_{n\rightarrow\infty}\frac{\sum_{f\in\hat{P}\setminus P_g}|H_f\cap M_n|}{|M_n|}\\
&\leq\sum_{f\in\hat{P}\setminus P_g}\limsup_{n\rightarrow\infty}\frac{|H_f\cap M_n|}{|M_n|}.
\end{align*}
It has already been noticed by Micheli and Schnyder \cite{reto} that one needs additional argumentation to show that the second inequality is really true, i.e. that the superior limit can be put into the series. As mentioned above (see Propositon \ref{gia}), in \cite{gm}, the author presents a way to prove \eqref{prop} in a more general setup.

For $m=1$, i.e. for the case of pairwise coprime polynomials, the problem can be fixed in a more elementary way employing the Lemma of Fatou \cite[p.82, p.89]{fitz} using the counting measure. To apply this lemma, one has to show the uniform convergence (in $n$) of 
$$\sum_{f\in\hat{P}\setminus P_g}\frac{|H_f\cap \tilde{M}_n|}{|\tilde{M}_n|}\quad \text{where}\quad \tilde{M}_n:=M_n\setminus\{d_i\equiv d_j\equiv 0\ \text{for some $i\neq j$}\}.$$
It is sufficient to consider $\tilde{M}_n$ since $(E_{P_g}\setminus E)\cap M_n=(E_{P_g}\setminus E)\cap \tilde{M}_n$ because $E\cap \{d_i\equiv d_j\equiv 0\ \text{for some $i\neq j$}\}=\emptyset$. One could assume $t^{-\deg(f)}\leq n$ since $f$ cannot divide non-zero polynomials of degree less than $\deg(f)$ and in $H_f$, there exist at least two polynomials $d_i, d_j$ with $i\neq j$ that are divided by $f$. One obtains
\begin{align*}
\frac{|H_f\cap \tilde{M}_n|}{|\tilde{M}_n|}&\leq\frac{|H_f\cap \tilde{M}_{\hat{n}}|}{|\tilde{M}_{\hat{n}}|}\cdot\frac{|\tilde{M}_{\hat{n}}|}{|\tilde{M}_n|}\\
&\leq (1-W_f)\cdot\frac{(n+t^{-\deg(f)}-r)^N+N\cdot(n+t^{-\deg(f)}-r)^{N-1}}{n^N+N\cdot n^{N-1}}\\
&\leq (1- W_f)\cdot \frac{(2n)^N+N\cdot(2n)^{N-1}}{n^N}\leq(1-W_f) \cdot( 2^N+N\cdot 2^{N-1}).
\end{align*}
Note that for $m> 1$, one only knows $t^{-\deg(f)}\leq n^m$ and one could not get a bound that is independent of $n$ like in the preceding estimations.

Since, for $m=1$, $W_f$ is equal to the probability that at most one of the polynomials $d_1,\hdots,d_N$ is divided by $f$, it holds
$W_f=(1-t^{\deg(f)})^N+Nt^{\deg(f)}(1-t^{\deg(f)})^{N-1}$.
As in the proof of Corollary \ref{gem} one has
\begin{align*}
(1-t^{\deg(f)})^N+Nt^{\deg(f)}(1-t^{\deg(f)})^{N-1}
=1-\binom{N}{2}t^{2\deg(f)}+\sum_{k=3}^{N}\alpha_kt^{k\cdot\deg(f)},
\end{align*}
with coefficients $\alpha_k\in \mathbb N$ that are independent of $\deg(f)$. It follows
$$1-W_f=\binom{N}{2}t^{2\deg(f)}+\sum_{k=3}^{N}\alpha_kt^{k\cdot\deg(f)}.$$
In summary, with $C(N):=2^N+N\cdot 2^{N-1}$, one gets
\begin{align*}
&\sum_{f\in\hat{P}\setminus P_g}\frac{|H_f\cap \tilde{M}_n|}{|\tilde{M}_n|}\leq C(N)\sum_{f\in\hat{P}\setminus P_g}\left(\binom{N}{2}t^{2\deg(f)}+\sum_{k=3}^{N}\alpha_kt^{k\cdot\deg(f)}\right)=\\
&=C(N)\sum_{j=g+1}^{\infty}\varphi_j\left(\binom{N}{2}t^{2j}+\sum_{k=3}^{N}\alpha_kt^{k\cdot j}\right)\\
&\leq C(N)\sum_{j=g+1}^{\infty}\binom{N}{2}t^{j}+\sum_{k=3}^{N}\alpha_kt^{(k-1)\cdot j},
\end{align*}
which converges uniformly in $n$ since the convergent bound is independent of $n$.

\end{remark}

\end{proof}
It remains to compute $W_j(N)$. To this end, we will firstly prove a recursion formula for it.

\begin{lemma}\label{reccon}\ \\
Let $\hat{A}$ be the set of matrices $K_i$ for which $\mathcal{K}_N$ has full row rank and $\det(K_i)=0$ for $i=1,\hdots,N$. Moreover, denote by $\hat{W}_j(N)$ the probability of $\hat{A}$. With $W_j(0)=W_j(1)=1$, it holds for $N\geq 2$:
\begin{align*}
W_j(N)&=\sum_{i=1}^{N}(-1)^{i-1}\binom{N}{i}\left(t^{jm^2}|GL_m(\mathbb F^j)|\right)^iW_j(N-i)+\hat{W}_j(N).
\end{align*}
\end{lemma}

\begin{proof}\ \\
If $\det(K_i)\neq 0$ for some $i\in\{1,\hdots,N\}$, it follows from Lemma \ref{ki} that $\mathcal{K}_{N}$  has full row rank if and only if the matrix $\mathcal{K}^{(i)}_{N-1}$ formed by the matrices from the set $\{K_1,\hdots,K_N\}\setminus\{K_i\}$ has full row rank. Using the inclusion-exclusion-principle with $\hat{A}$ and $A_i:=\{\det(K_i)\neq 0\ \text{and}\ \mathcal{K}^{(i)}_{N-1}\ \text{is of full row rank}\}$, where $\hat{A}\cap A_i=\emptyset$, $P_i:=\operatorname{Pr}(A_i)$ should denote the (uniform) probability of $A_i$ (under random choice of $K_i$), for $i=1,\hdots,N$ and $P_I:=\operatorname{Pr}\left(\bigcap_{i\in I}A_i(N)\right)$, one gets 
$$W_j(N)=\sum_{I\subset\{1,\hdots,N\}}(-1)^{|I|-1}P_I+\hat{W}_j(N).$$
Using Lemma \ref{ki}, one obtains\\
$\bigcap_{i\in I}A_i(N)=\{\det(K_i)\neq 0\ \text{for}\ i\in I\ \text{and}\ \mathcal{K}^{(I)}_{N-|I|}\ \text{has full row rank}\}$ and therefore, $P_I=\left(t^{jm^2}|GL_m(\mathbb F^j)|\right)^iW_j(N-i)$ for every $I$ with $|I|=i$. Since there are $\binom{N}{i}$ subsets of cardinality $i$, the formula follows.
\end{proof}

\begin{corollary}\label{mklein}\ \\
For $m\leq N-1$, it holds
$$W_j(N)=\sum_{i=1}^{N}(-1)^{i-1}\binom{N}{i}\left(t^{jm^2}|GL_m(\mathbb F^j)|\right)^iW_j(N-i).$$
\end{corollary}

\begin{proof}\ \\
If $\det(K_i)=0$ for $i=1,\hdots,N$, the column rank of $\mathcal{K}_N$ is at most $Nm-N<Nm-m=(N-1)m$ and therefore, one has no full row rank. 
Consequently, $\hat{W}_j=0$ and the statement follows from the preceding theorem.
\end{proof}

To obtain a formula for $W_j(N)$ in the general case, one finally needs to calculate $\hat{W}_j(N)$.

\begin{lemma}\ \\
For $j\in\mathbb N$ and $N\geq 2$, it holds:
$$\hat{W}_j(N)=\left(1-t^{jm^2}|GL_m(\mathbb F^j)|\right)^N-\sum_{i=N-1}^{\min(m,N-1)}t^{j(m+1)}+O(t^{(m+2)j}).$$
\end{lemma}



\begin{proof}\ \\
Denote by $\tilde{W}$ the probability that $\det(K_i)=0$ for $i=1,\hdots,N$ and $\mathcal{K}_N$ is not of full row rank. We will show 
$$\tilde{W}=\sum_{i=N-1}^{\min(m,N-1)}t^{(m+1)j}+O(t^{(m+2)j}).$$ 
The result follows since the sum of $\tilde{W}$ and $\hat{W}_j(N)$ is equal to the probability that $\det(K_i)=0$ for $i=1,\hdots,N$.

If $m<N-1$, the probability that $\det(K_i)=0$ for $i=1,\hdots,N$ is equal to $\left(1-t^{jm^2}|GL_m(\mathbb F^j)|\right)^N=(1-(1-t^j+O(t^{j+1}))^N=O(t^{jN})=O(t^{j(m+2)})$, which is conform with $\sum_{i=N-1}^{\min(m,N-1)}t^{j(m+1)}=0$ in this case.

Next, consider the case $m\geq N-1$. We have to compute the probability that there exists $\xi\in(\mathbb F^j)^{1\times m(N-1)}\setminus\{0\}$ with $\xi\mathcal{K}_N=0$, i.e. that there exist $\xi_i\in(\mathbb F^j)^{1\times m}$ for $i=1,\hdots,N-1$ which are not all identically zero such that $\xi_1K_1=(\xi_1+\xi_2)K_2=\cdots=(\xi_{N-2}+\xi_{N-1})K_{N-1}=\xi_{N-1}K_N=0$.
As in the proof for Lemma 10 of \cite{l}, one could show that either $\xi_i\neq 0$ for $i=1,\hdots,N-1$ and $\xi_i+\xi_{i+1}\neq 0$ for $i=1,\hdots,N-2$ or there exists $i\in\{1,\hdots,N\}$ such that $\mathcal{K}^{(i)}_{N-1}$ formed by the matrices from the set $\{K_1,\hdots,K_N\}\setminus\{K_i\}$ is not of full row rank. Per induction with respect to $N$, one knows that the probability for this is $O(t^{j(m+1)})$. Multiplication with the probability that $\det(K_i)=0$, which is $O(t^j)$, leads to a term for the probability that is $O(t^{j(m+2)})$. Note that one could use induction since for $N=2$, $\hat{W}$ is just equal to $1-t^{j\cdot 2m^2}N(2m,m,m)=1-\prod_{l=m+1}^{2m}(1-t^{jl})$ (see Lemma \ref{glcar}) because that $[K_1\ K_2]$ is not of full row rank already implies $\det(K_1)=\det(K_2)=0$. Thus, one could assume $\xi_i\neq 0$ and $\xi_i+\xi_{i+1}\neq 0$.

According to Lemma \ref{glcar}, the probability that $\dim(\ker(K_i))=r_i$ is equal to 
\begin{align*}
t^{jm^2}\cdot N(m,m,m-r_i)&=t^{jmr_i}\cdot\prod_{l=r_i+1}^{m}(1-t^{jl})\prod_{l=0}^{m-(r_i+1)}\frac{t^{j(l-m)}-1}{t^{-j(l+1)}-1}=\\
&=t^{jmr_i}(1+O(t^j))\cdot \frac{\prod_{l=r_i+1}^{m}(t^{-jl}-1)}{\prod_{l=1}^{m-r_i}(t^{-jl}-1)}=\\
&=t^{jmr_i}(1+O(t^j))\cdot t^{-\frac{j}{2}\left(m(m+1)-r_i(r_i+1)-(m-r_i)(m-r_i+1)\right)}\\
&=t^{jr_i^2}(1+O(t^j)).
\end{align*}
Fix $1\leq r_i\leq m$ for $i=1,\hdots,N$. Then, the probability that $\dim(\ker(K_1))=r_1$ is $t^{jr_1^2}\cdot (1+O(t^j))$. For each such matrix $K_1$, there are $t^{-jr_1}$ possibilities for $\xi_1\in(\mathbb F^j)^{1\times m}$ with $\xi_1K_1=0$. Furthermore, the probability that $\dim(\ker(K_2))=r_2$ is $t^{jr_2^2}\cdot (1+O(t^j))$ and for fixed $\xi_1$ and $K_2$, there are $t^{-jr_2}$ possibilities for $\xi_2\in(\mathbb F^j)^{1\times m}$ such that $(\xi_1+\xi_2)K_2=0$. This procedure is continued until $K_i$ and $\xi_i$ are fixed for $i=1,\hdots,N-1$.
As we assumed $\xi_{N-1}\neq 0$, the probability that $\xi_{N-1}K_N=0$ is equal to $t^{jm}$.

Finally, one has to consider, which values for $\xi_1,\hdots,\xi_{N-1}$ lead to the same solutions for $K_1,\hdots,K_N$. One clearly gets the same solutions if one multiplies $\xi_i$ for $i=1,\hdots,N-1$ by the same scalar value, which effects a factor that is $O(t^j)$ for the probability. In summary, the overall probability is $O\left(t^{j\left(\sum_{i=1}^{N-1}(r_i^2-r_i)+m+1\right)}\right)(1+O(t^j))$. Hence, all cases in which $r_i\geq 2$ for some $i\in\{1,\hdots,N-1\}$ could be neglected.

It remains to show that for $r_1=\cdots=r_N=1$, only $\xi_1,\hdots,\xi_{N-1}$ which differ all by the same scalar factor lead to the same solutions for $K_1,\hdots,K_N$. Then, one knows that the factor for the probability caused by this effect is exactly $t^j$ and one gets a overall probability of $t^{(m+1)j}+O(t^{(m+2)j})$, which is conform with $\sum_{i=N-1}^{\min(m,N-1)}t^{j(m+1)}=t^{j(m+1)}$ in the considered case $m\geq N-1$.

To do this, we firstly show that the case that $\xi_1,\hdots,\xi_{N-1}$ are linearly dependent could be neglected. For the choice of such vectors $\xi_i$ with the property that $\operatorname{rk}[\xi_1^\top\cdots \xi_{N-1}^\top]<N-1$ one has\\
$O\left(\sum_{r=1}^{N-2} N(m,N-1,r)\right)=O\left( \sum_{r=1}^{N-2}t^{-jr(m+N-1-r)}\right)=O(t^{-j(N-2)(m+1)})$ possibilities and for each of these possibilities the probability that $\xi_1K_1=(\xi_1+\xi_2)K_2=\cdots=(\xi_{N-2}+\xi_{N-1})K_{N-1}=\xi_{N-1}K_N=0$ is equal to $t^{jNm}$ as $\xi_i\neq 0$ and $\xi_i+\xi_{i+1}\neq 0$. Additionally, one has again a factor of $O(t^j)$ because of the values for the vectors $\xi_i$ that lead to the same solutions for $K_1,\hdots,K_N$. In summary, one gets a probability that is $O(t^{j(Nm+1-(N-2)(m+1))})=O(t^{j(m+2)})$ since $-N\geq -m-1$.

Hence, in the following, one could assume that $\xi_1,\hdots,\xi_{N-1}$ are linearly independent.
If $\xi_1K_1=\tilde{\xi}_1K_1=0,(\xi_1+\xi_2)K_2=(\tilde{\xi}_1+\tilde{\xi}_2)K_2=0, \hdots, \xi_{N-1}K_N=\tilde{\xi}_{N-1}K_N=0$, it results from $r_1=\cdots=r_N=1$ that there exist $\lambda_i\in\mathbb F^j$ with $\tilde{\xi}_1=\lambda_1\xi_1$, $\tilde{\xi_i}+ \tilde{\xi}_{i+1}=\lambda_{i+1}(\xi_i+\xi_{i+1})$ for $i=1,\hdots,N-2$ and $\tilde{\xi}_{N-1}=\lambda_{N}\xi_{N-1}$. Since 
 $\tilde{\xi}_1-(\tilde{\xi}_1+\tilde{\xi}_2)+\cdots \pm (\tilde{\xi}_{N-2}+\tilde{\xi}_{N-1})\mp \tilde{\xi}_{N-1}=0$, it follows $(\lambda_1-\lambda_2)\xi_1+(\lambda_3-\lambda_2)\xi_2+\cdots\pm(\lambda_{N-1}-\lambda_N)\xi_{N-1}=0$. As $\xi_1,\hdots,\xi_{N-1}$ are linearly independent, this implies $\lambda_1=\cdots=\lambda_N$, which completes the proof of the whole theorem.
\end{proof}

Now, we are able to solve the recursion formula of Lemma \ref{reccon} to achieve an explicit expression for $W_j(N)$.

\begin{theorem}\ \\
For $j\in\mathbb N$ and $N\geq 2$, the probability that $N$ constant matrices from $(\mathbb F^j)^{m\times m}$ are mutually left coprime is equal to
$$W_j(N)=1-\sum_{y=2}^{m+1}\binom{N}{y}t^{j(m+1)}+O(t^{j(m+2)}).$$
\end{theorem}

\begin{proof}\ \\
This is shown per induction with respect to $N$. For $N=2$, one just has to compute the probability that a rectangular matrix is of full rank. According to Lemma \ref{glcar} with $n=2m$ and $k=r=m$, this probability is equal to $\prod_{i=m+1}^{2m}(1-(t^{j})^i)=1-t^{j(m+1)}+O(t^{j(m+2)})$.

Inserting the assumption of the induction into the first part of the recursion formula from Lemma \ref{reccon}, leads to
\begin{align*}
&\sum_{i=1}^{N}(-1)^{i-1}\binom{N}{i}\left(t^{jm^2}|GL_m(\mathbb F^j)|\right)^iW_j(N-i)=\\
&=\sum_{i=1}^{N}(-1)^{i-1}\binom{N}{i}\left(t^{jm^2}|GL_m(\mathbb F^j)|\right)^i\left(1-\sum_{y=2}^{m+1}\binom{N-i}{y}t^{j(m+1)}+O(t^{j(m+2)})\right)\\
&=\sum_{i=1}^{N}(-1)\binom{N}{i}\left((-1)t^{jm^2}|GL_m(\mathbb F^j)|\right)^i+\\
&+\sum_{i=1}^{N}(-1)^{i}\binom{N}{i}\sum_{y=2}^{m+1}\binom{N-i}{y}t^{j(m+1)}+O(t^{j(m+2)})=\\
&=-\left(1-t^{jm^2}|GL_m(\mathbb F^j)|\right)^N+1+\\
&+\sum_{i=1}^{N-2}(-1)^{i}\binom{N}{i}\sum_{y=2}^{m+1}\binom{N-i}{y}t^{j(m+1)}+O(t^{j(m+2)})=\\
&=-\left(1-t^{jm^2}|GL_m(\mathbb F^j)|\right)^N+1-\sum_{y=2}^{\min(m+1,N-1)}\binom{N}{y}t^{j(m+1)}+O(t^{j(m+2)})
\end{align*}
since 
\begin{align*}
&\sum_{i=1}^{N-2}(-1)^{i}\binom{N}{i}\sum_{y=2}^{m+1}\binom{N-i}{y}=\sum_{i=1}^{N-2}\sum_{y=2}^{\min(m+1,N-i)}(-1)^i\frac{N!}{i!\cdot y!\cdot (N-i-y)!}\\
&=\sum_{y=2}^{\min(m+1,N-1)}\sum_{i=1}^{N-y}(-1)^i\frac{N!}{i!\cdot y!\cdot (N-i-y)!}=-\sum_{y=2}^{\min(m+1,N-1)}\binom{N}{y},
\end{align*}
where the last step follows from Lemma \ref{bin}.
Using the formula for $\hat{W}_j(N)$ from the preceding lemma, one obtains
\begin{align*}
W_j(N)&=-\left(1-t^{jm^2}|GL_m(\mathbb F^j)|\right)^N+1-\sum_{y=2}^{\min(m+1,N-1)}\binom{N}{y}t^{j(m+1)}+\hat{W}_j(N)\\
&=1-\left(\sum_{i=N-1}^{\min(m,N-1)}1+\sum_{y=2}^{\min(m+1,N-1)}\binom{N}{y}\right)t^{j(m+1)}+O(t^{j(m+2)})=\\
&=1-\sum_{y=2}^{m+1}\binom{N}{y}t^{j(m+1)}+O(t^{j(m+2)}).
\end{align*}
The last equality is valid because if $m+1\leq N-1$, it holds $\sum_{i=N-1}^{\min(m,N-1)}1=0$ and if $m+1>N-1$, it holds\\ $\sum_{i=N-1}^{\min(m,N-1)}1+\sum_{y=2}^{\min(m+1,N-1)}\binom{N}{y}=1+\sum_{y=2}^{N-1}\binom{N}{y}=\sum_{y=2}^{N}\binom{N}{y}=\sum_{y=2}^{m+1}\binom{N}{y}$.
%
\end{proof}

\begin{theorem}\label{main}\ \\
The natural density of $D_i\in\mathbb F[z]^{m\times m}$ for $i=1,\hdots,N$ to be mutually left coprime is equal to
$$\prod_{j=1}^{\infty}\left(1-\sum_{y=2}^{m+1}\binom{N}{y}t^{j(m+1)}+O(t^{j(m+2)})\right)^{\varphi_j}=1-\sum_{y=2}^{m+1}\binom{N}{y}t^m+O(t^{m+1}).$$
\end{theorem}

%
%
%

\section{Conclusion}
We computed the natural density of mutually left coprime polynomial matrices and compared the result with the uniform probability of mutual left coprimeness. If the considered matrices are scalar, i.e. for the case of pairwise coprime polynomials, we could even show a more precise estimation than in the general case.
It is remarkable that probability and natural density asymptotically coincide in all considered cases. However, the exact values for these two concepts of probability might differ. For the case of pairwise coprimeness of scalar polynomials, we have already seen that the coefficient of $t^2$ depends on the degrees of the constituent polynomials and is different from the coefficient of $t^2$ in the series expansion of the formula for the natural density if $N_1\neq 0$. Moreover, it is not difficult to see that further coefficients will also depend on the degrees of the involved polynomials.

For $m\geq 2$, the exact value for the uniform probability depends on the degrees of the determinants of the constituent matrices and therefore, does not coincide with the natural density for each degree structure.
Consider for example the case $m=2$ and $\deg(\det(D_i))=1$ for $i=1,2$. Easy computation yields that the uniform probability of left coprimeness is equal to $1-\frac{1}{t^{-2}+t^{-1}}=1-t^2\sum_{k=0}^{\infty}(-t)^k$, which is larger than the natural density being equal to $(1-t^2)(1-t^3)$. One could expect that with increasing the values $n_i$, the number of coinciding coefficients between uniform probability and natural density increases. But it is still an open question if the uniform probability of mutual left coprimeness tends to the value of the natural density if $n_i\rightarrow\infty$ for $i=1,\hdots,N$.


\bibliography{mybibfile}

\end{document}